\documentclass[a4paper,11pt]{article}
\usepackage[english]{babel}
\usepackage[latin1]{inputenc}
\usepackage[T1]{fontenc}
\usepackage{amsthm}
\usepackage{amsmath}
\usepackage{amssymb}
\usepackage{amscd} 
\usepackage{xypic}
\usepackage{mathrsfs}
\usepackage{braket}
\usepackage{xcolor}
\usepackage{hyperref}
\usepackage{graphicx,subfig,floatrow}
\usepackage[nottoc]{tocbibind}
\usepackage{caption}
\usepackage{tikz}
\usetikzlibrary{arrows,shapes,positioning,mindmap,trees,automata,decorations.markings}

\hypersetup{
	colorlinks=false,
	allbordercolors=white,
}

\theoremstyle{definition}
\newtheorem{definition}{Definition}[section]

\theoremstyle{plain}
\newtheorem{teo}[definition]{Theorem}
\newtheorem{prop}[definition]{Proposition}

\newtheorem{lem}[definition]{Lemma}

\newtheorem{prob}[definition]{Problem}

\newtheorem{claim}{Claim}

\newcommand{\numberset}{\mathbb}
\newcommand{\N}{\numberset{N}}

\newcommand{\Z}{\numberset{Z}}
\newcommand{\ca}{\mathcal}

\title{Highly edge-connected regular graphs without large factorizable subgraphs}

\author{Davide Mattiolo\thanks{Department of Physics, Informatics and Mathematics,
				University of Modena and Reggio Emilia, Via Campi 213/b, 41126 Modena, Italy. Email:~davide.mattiolo@unimore.it}, 	
Eckhard Steffen\thanks{Paderborn Center for Advanced Studies and Institute for Mathematics, Paderborn University, 
Warburger Str. 100, 33098 Paderborn, Germany. Email:~es@upb.de}}

\date{}

\begin{document}
	\maketitle
	
\begin{abstract}
We construct highly edge-connected $r$-regular graph which do not contain $r-2$ pairwise disjoint perfect matchings. 
The results partially answer a question stated by Thomassen \cite{Thomassen}.  
\end{abstract}

\section{Introduction}

We consider finite graphs which may have parallel edges but no loops. 
Let $r \geq 0$ be an integer. A graph $G$ is $r$-regular, if every vertex has degree $r$. For $1 \leq k \leq r$,
a graph $H$ is a $k$-factor of $G$, if $H$ is a spanning $k$-regular subgraph of $G$.  
Recently, Thomassen stated the following problem.

\begin{prob}[\cite{Thomassen}]\label{Prob:Thomassen}
Is every $r$-regular $r$-edge-connected graph of even order the union of $r-2$ $1$-factors and a $2$-factor?
\end{prob}

The statement is true for $r=3$. An $r$-regular graph $G$ is an $r$-graph, if $|\partial_G(S)| \geq r$ for every $S \subseteq V(G)$
with $|S|$ odd; $\partial_G(S)$ denotes the set of edges with precisely one end in $S$. We write $\partial(S)$ instead of $\partial_G(S)$ when it is clear the graph we are considering.

An $r$-graph is poorly matchable, if it
does not contain two disjoint 1-factors.
Clearly, every bridgeless cubic graph with edge chromatic number 4 is poorly matchable.  
Rizzi \cite{Rizzi} constructed poorly matchable $r$-graphs for each $r \geq 4$. All of them contain an edge of multiplicity $r-2$ and therefore, they have a 4-edge cut. However, the poorly matchable 4-graphs are 4-edge-connected and therefore, they provide a negative answer to Problem \ref{Prob:Thomassen}. The 4-graphs constructed in \cite{Mazz} are also poorly matchable. The following theorem is the main result of this note and it provides a negative answer to the question of Problem \ref{Prob:Thomassen} for every
positive integer which is a multiple of 4.

\begin{teo} \label{main1}
Let $t, r$ be positive integers and $r \geq 4$. 
There are infinitely many $t$-edge-connected $r$-graphs
which do not contain $r-2$ pairwise disjoint 1-factors, where \begin{itemize}
	\item $t=r$, if $r \equiv 0 \mod 4$;
	\item $t=r-1$, if $r \equiv 1 \mod 2$;
	\item $t=r-2$, if $r \equiv 2 \mod 4$.
\end{itemize}
\end{teo}

Indeed, we prove that for any $r \geq 4$, there are $r$-graphs of order 60 and simple $r$-graphs of order $70(r-1)$
with this properties.

\section{Proof of Theorem \ref{main1}}

Let $G$ be a graph and $N_1, \dots, N_k$ be a collection of subsets of $E(G)$. 
The graph $G' = G+(N_1 + \dots + N_k)$ consists of all vertices of $G$, 
i.e.~$V(G')=V(G)$, and $E(G')=E(G)\cup \bigcup_{i=1}^k N_i'$, where $N_i'$ is a copy of $N_i$. 

\subsection{Perfect matchings of the Petersen graph}

The edge set of a 1-factor of a graph $G$ is a perfect matching of $G$.  
A collection $\ca C$ is a set of objects where repetitions are allowed. Namely we can formally define it as a set $C=\{C_1,\dots, C_n\}$ together with a function $m\colon C \to \N $ which gives the multiplicity of each object $C_j$ in $\ca C$, that is the number of occurrences of $C_j$ in $\ca C$. A subcollection $\ca C'$ of $\ca C$ is a subset $C'\subseteq C$ with a function $m'\colon C'\to \N$ such that $m'(C_j)\le m(C_j)$ for all $j\in\{1,\dots,n\}$. In this case we will write $\ca C' \subseteq  \ca C.$ 

Our construction heavily relies on the properties of the Petersen graph $P$. Let $v_0\dots v_5$ and $u_1u_3u_5u_2u_4$ be the two disjoint $5$-cycles of $P$ such that $u_1v_1$, $u_2v_2$, $u_3v_3$, $u_4v_4, u_5v_5\in E(P) $. 
Let $M_0 = \{u_iv_i\colon i \in \{1,\dots,5\}\}$ and let $M_1$ be the only other perfect matching containing $u_1v_1$. Moreover let $M_i = \{u_{i+1}v_{j+1}\colon u_iv_j\in M_1\} $, where the sum of indices is taken modulo $5$, see Figure \ref{Fig:match_P}.  Let $\ca M = \{N_1,\dots,N_k\}$ be a collection of perfect matchings of $P$ and let $P^{\ca M}=P+\sum_{j=1}^k N_j$. We say that a perfect matching $N$ of $P^{\ca M}$ is of type $j$, if $N$ is a copy of $M_j$. In this case we write $t(N)=j$.
We will use the following results of Rizzi \cite{Rizzi}.

\begin{prop}[\cite{Rizzi}]\label{prop:bijection}
	The function associating to every pair of perfect matchings $M_i,M_j$ of $P$ the unique edge $e\in M_i\cap M_j$ is a bijection.
\end{prop}

\begin{lem}[\cite{Rizzi}]\label{lem:P_1}
	Consider a perfect matching $M_j$ of $P$ and let $P' = P+M_j$. Furthermore let $N_1$ and $N_2$ be two disjoint perfect matchings of $P'$. Then $j\in \{t(N_1),t(N_2)\}$.
\end{lem}

	\begin{lem}\label{lem:P+matchings}
Let $\ca M$ be a collection of $k$ perfect matchings of $P$. If $\ca M' = \{M_1',\dots,M_k',M_{k+1}'\}$ is a 
collection of $k+1$ pairwise disjoint perfect matchings of $P^{\ca M}$, then $ \ca M \subseteq \ca M'$.
	\end{lem}

	\begin{proof}
		We argue by induction over $k\in \N$. If $k=1$, then the statement holds by Lemma \ref{lem:P_1}. So let $k\ge 2$ 
		and $\ca M' = \{M_1',\dots,M_{k+1}',M_{k+1}'\}$ be pairwise disjoint perfect matchings of $P^{\ca M}$.
		
		If $M'_i=M'_j$ for all $i,j \in \{1,\dots,k+1\}$, then $\ca M$ must contain a unique perfect matching repeated $k$ times, and so $\ca M \subseteq \ca M'$.
		
		Otherwise there are $i,j\in\{1,\dots,k+1\}$ such that $M_i'\ne M_j'$. There is a unique edge $e\in P$ such that $\{e\} = M_i'\cap M_j'$. Such an edge must be a multiedge in $P^{\ca M}$. Then either $M_i'$ or $M_j'$ has been added to $P$ in order to obtain $P^{\ca M}$. This means that both $\ca M$ and $\ca M'$ contain a copy $M$ of the same perfect matching. Therefore, by the inductive hypothesis, $\ca M \setminus \{M\} \subseteq \ca M' \setminus \{M\}$, and so $\ca M \subseteq \ca M'.$
	\end{proof}

\subsection{$4k$-edge-connected $4k$-graphs without $4k-2$ pairwise disjoint perfect matchings}

	\begin{figure}
		\centering
		\includegraphics[scale=0.6]{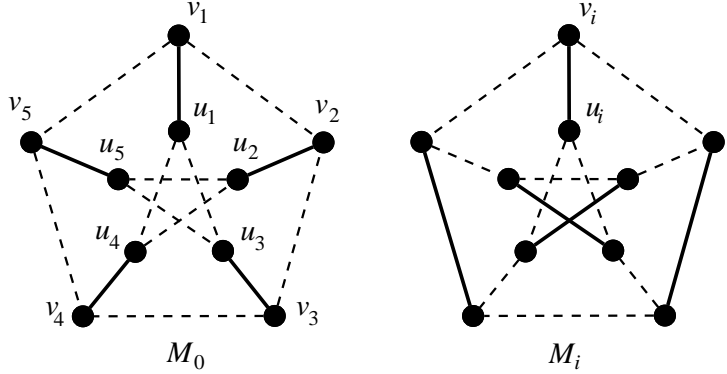}
		\caption{The perfect matchings $M_0,\dots, M_5$ of $P$.}\label{Fig:match_P}
	\end{figure}
	
For $k \geq 1$, let $P_k=P+kM_0+(k-1)(M_1+M_3+M_4)$.
	
	\begin{lem}\label{lem:P_k_connectivity}
For all $k \geq 1:$ $P_k$ is $4k$-edge-connected and $4k$-regular.
	\end{lem}

	\begin{proof} By definition, $P_k$ is $4k$-regular. 
		Let $X\subseteq V(P)$. If $|X|$ is odd, then every perfect matching intersects $\partial (X)$.
		Hence, $|\partial (X)|\ge 3 + k + 3(k-1) = 4k.$ If $|X|$ is even, then it suffices to consider the cases 
		$|X|\in\{2,4\}$. Since $P$ has girth $5$, the subgraph induced by $X$ is a path $P_X$ on either $2$ or $4$ vertices, having some multiple edge. In both cases, since the maximum multiplicity of an edge is $2k$, we have that $\partial(X)$ contains $2k$ edges per both end vertices of $P_X$, namely $|\partial(X)|\ge 4k $.
	\end{proof}

	Consider two copies $P_k^1$ and $P_k^2$ of $P_k$. If $u$ is a vertex (or edge) of $P_k$, then we denote $u^i$ the corresponding vertex (or edge) inside $P_k^i$. For $i \in \{1,2\}$ remove the multiedge $u_1^iv_1^i$ from $P_k^i$ and let $Q_k$ be the graph obtained by identifying $u_1^1$ and $u_1^2$ to the (new) vertex $u_{Q_k}$. It holds
	$d_{Q_k}(u_{Q_k}) = 4k$, and the $2k$ edges of $\partial(u_{Q_k})$ which are incident to vertices of $V(P_k^i)$ are denoted by $U_k^{i}$.
	If $Q_k$ is a subgraph of a graph $G$, then let 
		$V_k^i=\{xv_1^i \in E(G)\colon x\notin V(Q_k)\}$.
	
	\begin{figure}
		\centering
		\includegraphics[scale=0.55]{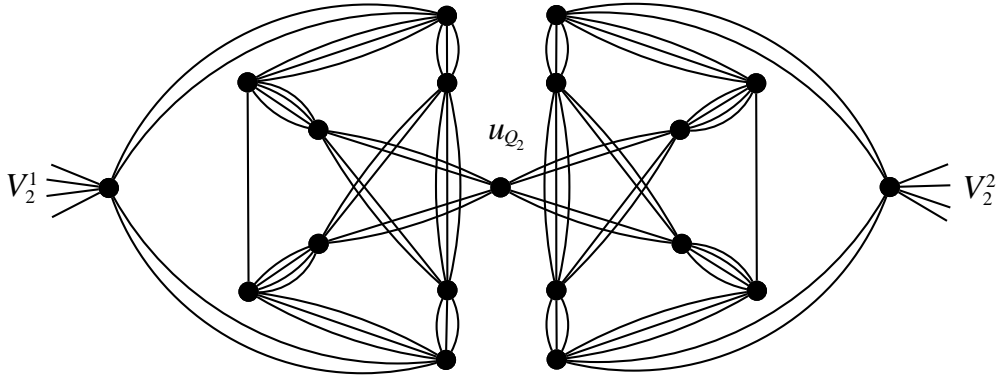}
		\caption{The subgraph $Q_2$.}\label{Fig:Q2}
	\end{figure}

	Let $\{N_1,\dots,N_n\}$ be a collection of perfect matchings of a graph $G$, define the function $\phi \colon E(G)\to \Z_2^n, e\mapsto (\phi_1(e),\dots,\phi_n(e)) $ such that $$\phi_j(e) = \begin{cases}
	1 & \text{if } e\in N_j;\\
	0 & \text{otherwise.}
	\end{cases}$$
	Moreover, if $W\subseteq E(G)$, then let $\phi(W)=\sum_{e\in W}\phi(e)$.
	For a vector $x = (x_1,\dots,x_n) \in \Z_2^n$ the number of its non-zero entries is denoted by $\omega(x)$.
	
	\begin{lem}\label{lem:Qk}
		Let $Q_k$ be a subgraph of a graph $G$ with $\partial(V(Q_k)) = V_k^{1} \cup V_k^{2}$. If, for all $i\in \{1,2\}, d_G(v_1^i)=4k$ and
		$\ca N =\{N_1,\dots, N_{4k-2}\}$ is a family of pairwise disjoint perfect matchings of $G$,
		then $$\omega(\phi(V^1_k))=\omega(\phi(V^2_k))=2k-1.$$
	\end{lem}

	\begin{proof}
Every perfect matching of $G$ intersects $\partial(V(Q_k))$ precisely once since $|V(Q_k)|$ is odd.

It remains to show that $V_k^i$ intersects precisely $2k-1$ elements of $\ca N$.
Recall that $Q_k$ is constructed by using two copies of $ P + \sum_{M\in \ca M} M$, where
$$\ca  M = \{M_0,\underbrace{M_0,M_1,M_3,M_4,\dots,M_0,M_1,M_3,M_4}_{(k-1)\text{-times}}\}.$$
 
Since $|V_k^{i}| = 2k$ and $\ca N$ contains $4k-2$ perfect matchings, it follows 
that $\omega(\phi(V_k^i))\in \{ 2k-2,2k-1,2k\}$.
Suppose to the contrary that $\omega(\phi(V_k^1))=2k-2$, which is equivalent to $\omega(\phi(V_k^2))=2k$.
Furthermore, $U_k^1$ ($U_k^2$) intersects the same matchings of $\ca N$ as $V_k^2$ ($V_k^1$). 
Therefore, there is a family $\ca N_P$ of $4k-2$ pairwise disjoint perfect matching in $P_k^{1}$ such that 
$\ca M \nsubseteq \ca N_P$, a contradiction to Lemma \ref{lem:P+matchings}. 
Hence, $\omega(\phi(V_k^1))= \omega(\phi(V_k^2))=2k-1$.
	\end{proof}
	
	Let $T_k$ be the graph on three vertices $x_1,x_2,x_3$ such that, for all $i\ne j$, there are $k$ parallel edges connecting $x_i$ to $x_j$.
	
	\begin{figure}
		\centering
		\includegraphics[scale=0.55]{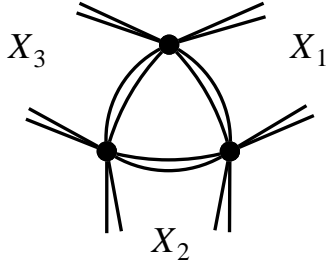}
		\caption{The subgraph $T_2$.}\label{Fig:T2}
	\end{figure}
	
Let $G$ be a cubic graph. Construct the graph $S_k(G)$ as follows: replace every node $v\in V(G)$ by a copy $T_k^v$ of the graph $T_k$ and every edge $e \in E(G)$ by a copy $Q_k^e$ of the graph $Q_k$. If the vertex $v'$ is adjacent with the edge $e'$, then the graphs $T_k^{v'}$ and $Q_k^{e'}$ are connected by $2k$ edges. More precisely, add $k$ edges connecting $v_1^1$ (or $v_1^2$) together with $x_i$ and $k$ edges connecting $v_1^1$ (or $v_1^2$) together with $x_{i+1}$, for suitable $i\in \Z_3$. Connect those graphs 
in such a way that the resulting graph $S_k(G)$ is $4k$-regular.
	
	Let $p=w_1e_1\dots w_ne_n$ be a path in $G$, for $w_j\in V(G)$ and $e_j\in E(G)$, then the chain $C=T_k^{w_1}Q_k^{e_1}\dots T_k^{w_n}Q_k^{e_n}$ consists of graphs which are connected to the previous and the next one, with respect to the chain order, in $S_k(G)$. In this case, we will say that the chain of a graph $C$ forms a path in $S_k(G)$.

	\begin{lem} \label{connectivity}
		Let $G$ be a bridgeless cubic graph. For all $k \geq 1:$
		$S_k(G)$ is a $4k$-edge-connected $4k$-regular graph. 
	\end{lem}

	\begin{proof}
$S_k(G)$ is $4k$-regular by construction. We show that 
there are $4k$ pairwise disjoint paths between any two vertices of $S_k(G)$. 
Consider the graph $R_k = P_k - u_1v_1$, where we remove from $P_k$ all 
($2k$) edges connecting $u_1$ to $v_1$.
	
		\begin{claim}\label{claim:edge-disj-paths}
			The following statements hold:
			\begin{itemize}
				\item[i.] there are $2k$ edge-disjoint $u_1v_1$-paths in $R_k$;
				\item[ii.] for all $w\in V(P_k)\setminus\{u_1,v_1\}$ there are $2k$ $wu_1$-paths and $2k$ $wv_1$-paths which are pairwise edge-disjoint in $R_k$;
				\item[iii.] for all $w_1\ne w_2\in V(P_k)\setminus\{u_1,v_1\}$, there exists $t\in\{0,1,\dots, 2k\}$ such that \begin{itemize}
					\item there are $t$ edge-disjoint $w_1w_2$-paths containing $u_1v_1$ in $P_k$;
					\item there are $4k-t$ edge-disjoint $w_1w_2$-paths in $R_k$, which are moreover edge-disjoint from the previous ones.
				\end{itemize}
				\item[iv.] for all $x_i\ne x_j\in V(T_k)$, there are $2k$ edge-disjoint $x_ix_j$-paths in $T_k$.
			\end{itemize}			
		\end{claim}
		\begin{proof}
			By Lemma \ref{lem:P_k_connectivity} there are $4k$ edge-disjoint $u_1v_1$-paths in $P_k$. Since $\mu (u_1v_1)=2k$ 
			there are $2k$ edge-disjoint $u_1v_1$-paths in $R_k$ and $i.$ is proved.
			
			Let $w\in V(P_k)\setminus\{u_1,v_1\}$. By Lemma \ref{lem:P_k_connectivity}, there are $4k$ edge-disjoint $wv_1$-paths 
			in $P_k$. Then, since $P_k$ is $4k$-regular and $u_1v_1$ is an edge of multiplicity $\mu(u_1v_1)=2k$, there are $2k$ of those paths ending with the edge $u_1v_1$. Thus, there are $2k$ $wu_1$-paths and $2k$ $wv_1$-paths which 
			are pairwise edge-disjoint in $R_k$ and so $ii.$ is proved.
			
			In order to prove statement $iii.$, pick two different vertices $w_1$ and $w_2$ in $V(P_k)\setminus\{u_1,v_1\}$. Since $P_k$ is $4k$-edge-connected there are $4k$-edge-disjoint $w_1w_2$-paths in $P_k$. Let $t$ be the number of such paths containing the edge $u_1v_1.$ Then $t\le \mu(u_1v_1)=2k$.
			
			The last statement holds since there are $k$ pairwise edge-disjoint paths $x_ix_j$ and 
			there are $k$ pairwise edge-disjoint paths $x_ix_tx_j$, for $t\ne i,j$. Thus, the claim is proved.
		\end{proof}
		
Let $y_1\ne y_2\in V(S_k(G))$. There are copies of $T_k$ or $Q_k$, say $Y_1, Y_2$, such that $y_i \in Y_i$.

Case 1: $Y_1$ and $Y_2$ correspond to two vertices $w_1$ and $w_2$ of $G$, that is, they both are copies of $T_k$.
If  $w_1 = w_2$, then $Y_1 = Y_2$ and the statement is trivial. 
If $w_1\ne w_2$, since $G$ is bridgeless, there are two edge-disjoint $w_1w_2$-paths in $G$. These paths correspond 
to two chains of (internally) different subgraphs $C=Y_1N_1\dots N_pY_2$ and $C'=Y_1N_{1}'\dots N_{q}'Y_2$ 
that both form a path in $S_k(G)$. Let $s_j,t_j$ the nodes of $N_j$ which are adjacent to $N_{j-1}$ and $N_{j+1}$ respectively. 
Let $s_1$ be adjacent to $Y_1$ and $t_p$ be adjacent to $Y_2$. Define in the very same way the vertices $s_j',t_j'$ in $C'$. 
By Claim \ref{claim:edge-disj-paths}, there are $2k$ pairwise edge-disjoint $s_1t_p$-paths 
in $C$ and $2k$ pairwise edge-disjoint $s_1't_q'$-paths 
in $C'$. Notice that $s_1\ne s_1'$ and $t_p\ne t_q'.$ By Claim \ref{claim:edge-disj-paths}, there are $2k$ pairwise 
edge-disjoint 
$s_1s_1'$-paths passing through $Y_1$ and $2k$ edge-disjoint $t_pt_q'$-paths passing through $Y_2$. Therefore, all these 
paths combine to $4k$ edge-disjoint $y_1y_2$-paths in $S_k(G)$.
		
Case 2: If $Y_1$ and $Y_2$ correspond to a vertex and an edge, or to two edges of $G$, say $a_1$ and $a_2$,
then there is a circuit in $G$ which contains $a_1$ and $a_2$. By a similar argumentation as above we deduce that there are  
$4k$-edge-disjoint $y_1y_2$-paths in $S_k(G)$.
	\end{proof}

\begin{teo} \label{teo:0mod4_general}
		Let $G$ be a bridgeless cubic graph with an even number of edges. For all $k \geq 1:$
		$S_k(G)$ is a $4k$-edge-connected $4k$-graph without $4k-2$ pairwise disjoint perfect matchings.
	\end{teo}

\begin{proof}
By Lemma \ref{connectivity}, $S_k(G)$ is $4k$-edge-connected, $4k$-regular and 
it holds $|V(S_k(G))|=19|E(G)|+3|V(G)|\equiv 0 \mod{2}$. Thus, $S_k(G)$ is a $4k$-graph. 

Suppose to the contrary that $S_k(G)$ has $4k-2$ pairwise disjoint perfect matchings. Consider a vertex $v\in V(G)$ and the corresponding subgraph $T_k^v$. Since $T_k^v$ has three vertices it follows that no component of $\phi(\partial_{S_k(G)}(V(T_k^v)))$ is $0$.
Let $\partial_G(v)=\{e_1,e_2,e_3\}$ and $X_i$ be the set of $2k$ edges connecting $T_k^v$ to $Q_k^{e_i}$, see Figure \ref{Fig:T2}. By Lemma \ref{lem:Qk}, for all $i\in \{1,2,3\}$, we have that $\omega(\phi(X_i))=2k-1$. Since the cardinality of the symmetric difference of three odd sets is odd it follows that there is a $j\in \{1,2,\dots,4k-2\}$ such that $0=\sum_{i=1}^3 \phi_j(X_i)=\phi_j(\partial_{S_k(G)}(V(T_k^v)))$, a contradiction.
\end{proof}

\subsection{Highly connected $r$-graphs on 60 vertices}		
 
To prove the other cases of Theorem \ref{main1}, we continue with the construction of regular graphs on 60 vertices.

For $k \geq 1$: Let $H_k$  be the graph which is obtained from three copies $Q_k^1, Q_k^2, Q_k^3$ of $Q_k$. In order to simplify the description let $z_j$ be the vertex $v_1^j$ of $Q_k$, for $j\in\{1,2\}$. For $i\in\{1,2,3\}$, if $u$ is a vertex of $Q_k$ we denote by $u^i$ the corresponding vertex of the copy $Q_k^i$. Glue them together with the graph $T_k$ as follows: for all $i\in\{1,2,3\}$,
\begin{itemize}
	\item add $k$ edges connecting $x_{i+1}$ of $T_k$ to $z_1^i$ of $Q_k^i$;
	\item add $k$ edges connecting $x_{i+2}$ of $T_k$ to $z_1^i$ of $Q_k^i$;
	\item add $k$ edges connecting $z_2^i$ of $Q_k^i$ to $z_2^{i+1}$ of $Q_k^{i+1}$;
\end{itemize} where the indices are added modulo 3.
The graph $H_2$ is depicted in Figure \ref{Fig:smallest}.

\begin{lem} \label{teo:0mod4}
For all $k \geq 1$: $H_k$ is a $4k$-edge-connected $4k$-graph of order 60 without $4k-2$ pairwise disjoint 
perfect matchings.   
\end{lem}

\begin{proof}
Let $K_2^3$ be the unique (loopless) cubic graph on two vertices. $H_k$ is obtained from $S_k(K_2^3)$ by removing one $T_k$ and 
then connecting the vertices of degree $2k$ pairwise by $k$ (parallel) edges. Clearly, $H_k$ is $4k$-regular.
Note that $\{z_2^1, z_2^2, z_2^3\}$ induce a triangle $T$ in $H_k$ where any two vertices are connected by $k$ edges. 
Furthermore, for any $2k$ pairwise edge-disjoint paths which connect two vertices of $\{z_2^1, z_2^{2}, z_2^{3}\}$
in $S_k(K_2^3)$ and do not contain any edge of $S_k(K_2^3) - T_k$ there are $2k$ corresponding paths in $H_k$. Hence,
$H_k$ is $4k$-edge-connected.

Suppose to the contrary that $H_k$ has $4k-2$ pairwise disjoint perfect matchings $\ca N = \{N_1,\dots,N_{4k-2}\}$. 
By Lemma \ref{lem:Qk},
for each $i \in \{1,2,3\}$, there are $2k-1$ edges of $E(T)$ which intersect an element of $\ca N$ and which
are incident to $z_2^{i}$, a contradiction.   
\end{proof}

\begin{figure}
	\centering
	\includegraphics[scale=0.27]{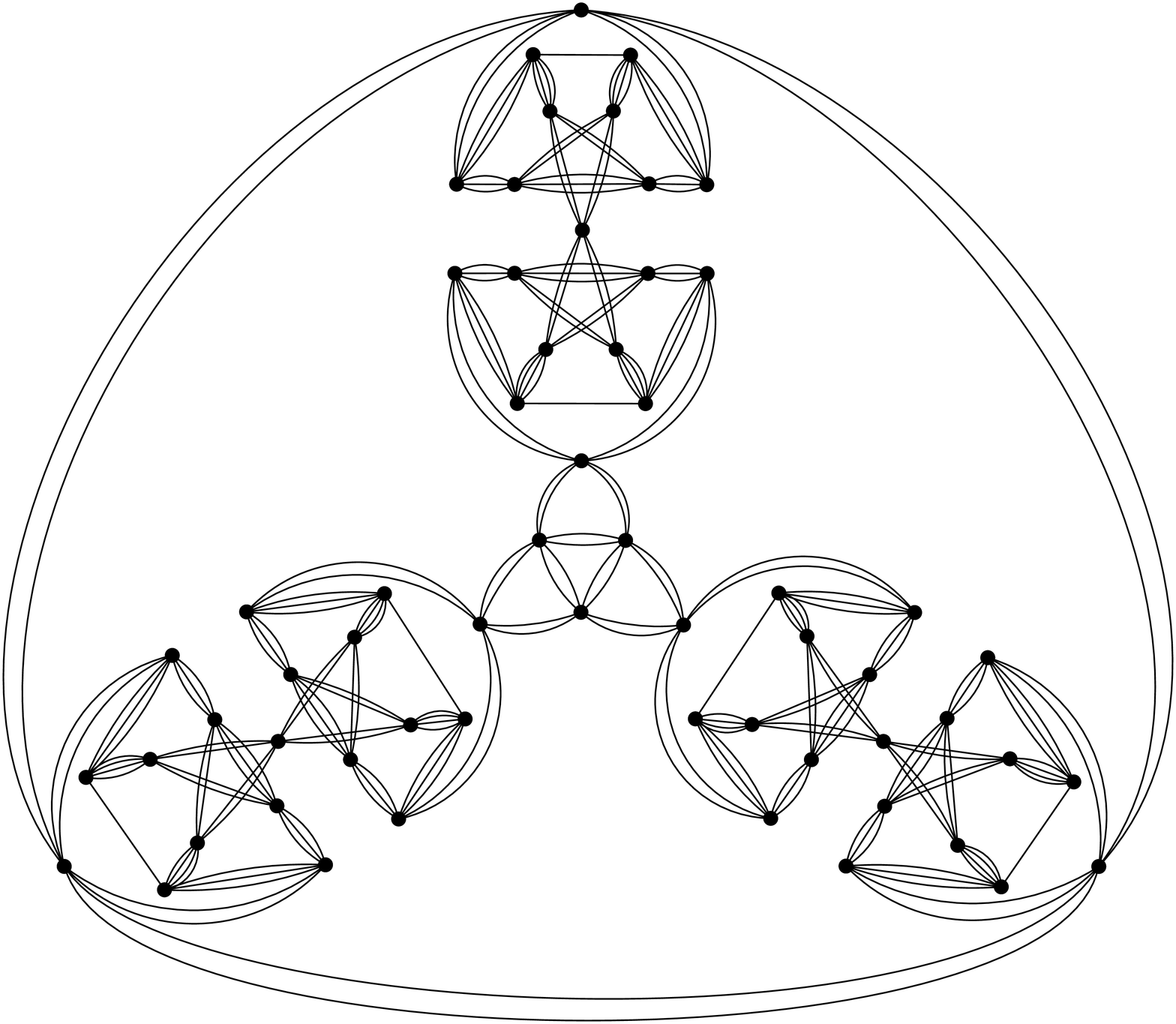}
	\caption{$H_2$ is an $8$-edge-connected $8$-graph on $60$ vertices without $6$ 
	pairwise disjoint perfect matchings.}\label{Fig:smallest}
\end{figure}

Next we will identify 4 pairwise disjoint matchings in $H_2$. These matchings will be used to complete the
proof of Theorem \ref{main1}.

Consider a copy of $Q_k$ inside a graph $G$, such that both $V_k^1$ and $V_k^2$ are non-empty. Let $M$ be a perfect matching of $G$. Then w.l.o.g. $|V_k^1\cap M|=1$ and $|V_k^2\cap M|=0$.
The unique perfect matching in $P_k=P_k^1+2ku_1^1v_1^1$ containing the edges of $M\cap E(P_k^1)$ is of type $0$ or $1$, suppose of type $0$. 
In the same way, the unique perfect matching in $P_k=P_k^2+2ku_1^2v_1^2$ containing the edges of $M\cap E(P_k^2)$ is of 
type $3$ or $4$, suppose of type $3$. In this case we say that $Q_k$ is of type $(0,3)$. For example, the bold perfect matching depicted in Figure \ref{Fig:4_p_matchings} is such that all $Q_k^i$s are of type $(0,4)$. We call $N_0$ such a perfect matching in $H_k$. Moreover, for $i\in\{1,2,3\}$, let $N_i$ be the perfect matching  of $H_{k}$ such that:
\begin{itemize}
		\item $Q_k^i$ is of type $(1,3)$;
		\item $Q_k^{i+1}$ is of type $(3,0)$;
		\item $Q_k^{i+2}$ is of type $(4,1)$;
	\end{itemize} where sums of indices are taken modulo $3$. In Figure \ref{Fig:4_p_matchings} $N_1$ is depicted using normal lines, $N_2$ is depicted using dotted lines and $N_3$ is depicted using dashed lines.

By construction of the perfect matchings $N_0,N_1,N_2, N_3$, the following lemma, which will be needed for the proof of Theorem \ref{main1},  follows.
\begin{lem}
	For all $k\ge1$: $ H_{k+1} = H_k + (N_0+N_1+N_2+N_3)$.
\end{lem}

\begin{lem}\label{teo:1_3mod4}
For all $t\ge1$, there is a $2t$-edge-connected $(2t+1)$-graph on 60 vertices without $2t-1$ pairwise disjoint perfect matchings.
\end{lem}

\begin{proof}
Case 1: $t=2k+1$ for a $k \geq 1$. Let $H_k'=H_k+(N_0+N_1+N_2)$. Since the graph $\tilde{H}=H_k[N_0+N_1+N_2]$ is a $3$-edge-colorable connected cubic graph, we have that for all $X\subseteq V(\tilde{H})$, $|\partial_{\tilde{H}}(X)|\ge3$, if $X$ is odd and $|\partial_{\tilde{H}}(X)|\ge2$, if $X$ is even. Then $H_k'$ is $(4k+2)$-edge-connected $(4k+3)$-graph. From the equality $H_k' = H_{k+1}-N_3$ we deduce that it has no $4k+1$ pairwise disjoint perfect matchings.

Case 2: $t = 2k$ for a $k \geq 1$. The graph $H_k''=H_k+N_0$ is a $4k$-edge-connected $(4k+1)$-graph. Since $H_k'' = H_{k+1}-(N_1+N_2+N_3)$, it follows that $H''_k$ has no $4k-1$ pairwise disjoint perfect matchings.
\end{proof}

\begin{figure}
	\centering
	\includegraphics[scale=0.28]{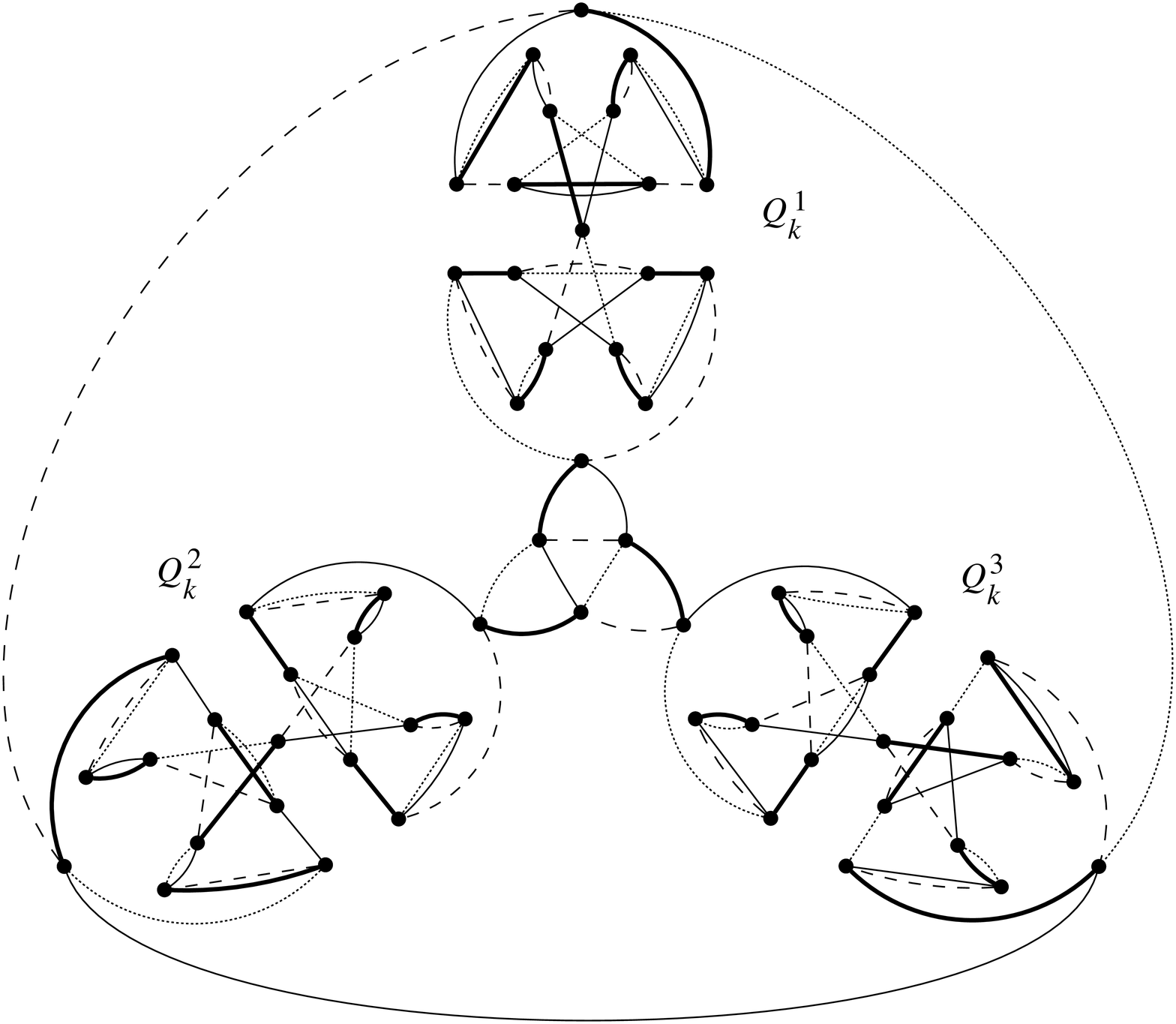}
	\caption{Four pairwise disjoint perfect matchings in $H_k$.}\label{Fig:4_p_matchings}
\end{figure}

\begin{lem}\label{teo:2mod4}
For all $k\ge1$, there is a $4k$-edge-connected $(4k+2)$-graph on 60 vertices without $4k$ pairwise disjoint perfect matchings.
\end{lem}

\begin{proof}
	The graph $H'''_k=H_k+N_0+N_1$ is a $4k$-edge-connected $(4k+2)$-graph. It has no $4k$ pairwise disjoint perfect matchings because $H'''_k=H_{k+1}-(N_2+N_3)$.	
\end{proof}

Theorem \ref{main1} now follows from Lemmas \ref{teo:0mod4}, \ref{teo:1_3mod4}, and \ref{teo:2mod4}. 
It remains to construct simple graphs with the desired property and to show how to expand vertices.  

\subsection{Simple graphs}

Let $v$ be a vertex of a graph $G$ such that $d_G(v)=t$. Moreover let $v_1,\dots,v_t$ be the not necessarily distinct neighbors of $v$ and $u_1,\dots,u_t$ be the vertices of degree $t-1$ of $K_{t,t-1}$. The Meredith extension \cite{Meredith} applied to $G$ at $v$ produces the graph $G_v$ obtained from $G-v$ and a copy of the complete graph $K_{t,t-1}$ by adding all edges $v_iu_i$, for $i\in\{1,\dots,t\}$. Notice that $G$ is $t$-edge-connected if and only if $G_v$ is $t$-edge-connected. Furthermore, it is easy to see that 
for $t \geq 2$, 
$G$ does not have $t$ pairwise disjoint perfect matchings if and only if $G_v$ does not have $t$ pairwise disjoint perfect matchings. 

Let $\ca V \subset V(H_k)$ be a vertex cover of $H_k$. If Meredith extension is applied on every vertex of $\ca V$, then
we obtain simple $r$-edge-connected $r$-graphs without $r-2$ pairwise disjoint perfect matchings.   
In particular, there is a vertex cover $\ca V$ of $H_k$ such that $|\ca V|=35$. Thus,
expanding the vertices of $\ca V$ at the graphs of Lemmas \ref{teo:0mod4}, \ref{teo:1_3mod4}, and \ref{teo:2mod4} 
yields simple $t$-edge-connected $r$-graphs of order $70(r - 1)$ with the desired properties. Repeated application of Meredith extension yields infinite families of such graphs.

\end{document}